\documentclass[12pt]{article}
\usepackage[usenames]{color}
\usepackage{graphicx}

\usepackage{amsmath,amssymb,amsthm,stmaryrd}
\usepackage[all]{xy}
\xyoption{arc}

\newtheorem{thm}{Theorem}
\newtheorem{cor}[thm]{Corollary}
\newtheorem{prop}[thm]{Proposition}
\newtheorem{lem}[thm]{Lemma}
\theoremstyle{definition}
\newtheorem{defn}[thm]{Definition}
\theoremstyle{remark}
\newtheorem{rmk}[thm]{Remark}

\def\co{\colon\thinspace}
\newcommand{\mb}[1]{\mathbb{#1}}

\newcommand{\colim}{\ensuremath{\mathop{\rm colim}}}
\newcommand{\hocolim}{\ensuremath{\mathop{\rm hocolim}}}

\newcommand{\Map}{\ensuremath{{\rm Map}}}

\newcommand{\Th}{{\rm Th}}

\newcommand{\Pic}{{\rm Pic}}
\newcommand{\Or}{{\rm Or}}
\newcommand{\pic}{{\rm pic}}

\newcommand{\hofib}{{\rm hofib}}

\newcommand{\eilm}[1]{\ensuremath{{H} #1}}
\newcommand{\smsh}[1]{\ensuremath{\mathop{\wedge}_{#1}}}
\newcommand{\tens}[1]{\ensuremath{\mathop{\otimes}_{#1}}}

\newcommand{\pow}[1]{\left\llbracket{#1}\right\rrbracket}

\newcommand{\Pm}{{\mb P}_m}
\newcommand{\Pp}{{\mb P}_p}
\newcommand{\Dm}{D_m}

\title{Strictly commutative complex orientation theory}
\author{Michael J. Hopkins\thanks{Partially supported by NSF grant
    DMS--0906194.}, Tyler Lawson\thanks{Partially supported by NSF
    grant DMS--1206008.}}

\begin{document}
\maketitle

\begin{abstract}
For a multiplicative cohomology theory $E$, complex orientations are
in bijective correspondence with multiplicative natural
transformations to $E$ from complex bordism cohomology $MU$. If $E$ is
represented by a spectrum with a highly structured multiplication, we
give an iterative process for lifting an orientation $MU \to E$ to a
map respecting this extra structure, based on work of Arone--Lesh. The
space of strictly commutative orientations is the limit of an inverse
tower of spaces parametrizing partial lifts; stage $1$ corresponds to
ordinary complex orientations, and lifting from stage $(m-1)$ to stage
$m$ is governed by the existence of an orientation for a family of
$E$-modules over a fixed base space $F_m$.

When $E$ is $p$-local, we can say more. We find that this tower only
changes when $m$ is a power of $p$, and if $E$ is $E(n)$-local the
tower is constant after stage $p^n$. Moreover, if the coefficient ring
$E^*$ is $p$-torsion free, the ability to lift from stage $1$ to stage
$p$ is equivalent to a condition on the associated formal group law
that was shown necessary by Ando.
\end{abstract}

Characteristic classes play a fundamental role in algebraic topology,
with the primary example being the family of Chern classes $c_i(\xi)
\in H^{2i}(X)$ associated to a complex vector bundle $\xi \to X$. Not
all generalized cohomology theories possess Chern classes, but they
are present in important cases such as complex $K$-theory $K$ and
complex bordism theory $MU$. In fact, for a cohomology theory $E$
taking values in graded-commutative rings, the following types of
information are equivalent:
\begin{itemize}
\item a choice of characteristic class $c_1(\xi) \in \widetilde
  E^2(X)$ for complex line bundles $\xi \to X$ such that, for the
  canonical line bundle $\gamma_1 \to \mb{CP}^1$, the isomorphism
  $\widetilde E^2(\mb{CP}^1) \cong E^0$ carries $c_1(\gamma_1)$ to 1;
\item a family of characteristic classes $c_i(\xi) \in \widetilde
  E^{2i}(X)$ for complex vector bundles $\xi \to X$ satisfying the
  above formula for $c_1(\gamma_1)$ and such that the Cartan formula
  for Whitney sums holds; or
\item a natural transformation $MU \to E$ of multiplicative cohomology
  theories.
\end{itemize}
In the third case, the natural transformation $MU^{2i}(X) \to
E^{2i}(X)$ allows us to push forward the characteristic classes
$c_i^{MU}(\xi)$ to classes $c_i^E(\xi)$. Such a map $MU \to E$ is
called a complex orientation of $E$, and as a result $MU$ plays a
fundamental role in the theory of Chern classes.

Moving from the homotopy category to the point-set level, the spectrum
$MU$ representing complex bordism is also one of the best known
examples of a spectrum with a multiplication which is associative and
commutative up to all higher coherences (an $E_\infty$ ring
structure). If we know that $E$ is also equipped with an $E_\infty$
ring structure, it is natural to ask whether a complex orientation
$MU \to E$ can be lifted to a map respecting this $E_\infty$ ring
structure (an $E_\infty$ orientation), and what data is necessary to
describe this. This is a stubborn problem and in prominent cases the
answer is unknown, such as when $E$ is a Lubin--Tate spectrum.

For any complex orientation of $E$, there is a formal group law $\mb G$
expressing the first Chern class of a tensor product of two complex
line bundles: we have
\[
c_1(\xi' \otimes_{\mb C} \xi'') = c_1(\xi') +_{\mb G} c_1(\xi'')
\]
for an associative, commutative, and unital power series
$x +_{\mb G} y \in E^*\pow{x,y}$. Ando gave a necessary and sufficient
condition for an orientation $MU \to E$ to be an $H_\infty$
orientation \cite{ando-thesis}, a weaker structure than an $E_\infty$
orientation which can be described as a natural transformation that
respects geometric power operations. Ando's condition was that a
certain natural power operation $\Psi$ with source $E^*(X)$ should act
as the canonical Lubin isogeny on the coordinate ring
$E^*(\mb{CP}^\infty)$ of the formal group law $\mb G$ (see also
\cite[\S 4.3]{ando-hopkins-strickland-sigma}). In general, this is
stronger than the data of a complex orientation alone
\cite{noel-johnson-ptypical}, and very few $E_\infty$ ring spectra are
known to admit $H_\infty$ orientations. (Ando also showed that
Lubin--Tate spectra associated to the Honda formal group law have
unique $H_\infty$ orientations; in recent work, Zhu has generalized
this to all of the Lubin--Tate spectra \cite{zhu-normcoherence}.)

However, it is the case that the rationalization $MU_{\mb Q}$ is
universal among rational, complex oriented $E_\infty$ rings
\cite[6.1]{baker-richter-thom}. Further, Walker studied orientations
for the case of $p$-adic $K$-theory and the Todd genus
\cite{walker-thesis}, and M\"ollers studied orientations in the case
of $K(1)$-local spectra \cite{mollers-thesis}. Both gave proofs that
Ando's condition for $H_\infty$ orientations was also sufficient to
produce $E_\infty$ orientations.

The goal of this paper is to apply work by Arone--Lesh
\cite{arone-lesh-filtered} to extend this procedure, giving an
inductive approach to the construction of $E_\infty$
orientations. Before getting into details, we will describe the
motivation for this construction.

As an $E_\infty$ ring spectrum, the fact that $MU$ is a Thom spectrum
gives it a universal property. There is a map of infinite loop spaces
$U \to GL_1(\mb S)$ from the infinite unitary group to the space of
self-equivalences of the sphere spectrum $\mb S$, and $MU$ is
universal among $E_\infty$ ring spectra $E$ with a chosen nullhomotopy
of the map of infinite loop spaces $U \to GL_1(\mb S) \to GL_1(E)$
\cite[\S V]{may-quinn-ray-ringspectra}.

We can recast this using the language of Picard groups. We consider
two natural functors: one sends a complex vector space $V$ to the
spectrum $\Sigma^\infty S^V$, which has an inverse under the smash
product; the second sends a smash-invertible spectrum $I$ to a
smash-invertible $MU$-module $MU \wedge I$. On restricting to the
subcategory of weak equivalences and applying classifying spaces, we
obtain maps of $E_\infty$ spaces
\[
\coprod_m BU(m) \to \mb Z \times BGL_1(\mb S) \to \mb Z \times BGL_1(MU).
\]
Here the latter two are the Picard spaces $\Pic(\mb S)$ and $\Pic(MU)$
\cite[2.2.1]{mathew-stojanoska-pictmf}. Passing through an infinite
loop space machine, we obtain a sequence of maps of spectra
\[
ku \to \pic(\mb S) \to \pic(MU),
\]
where $ku$ is the connective complex $K$-theory spectrum. The
universal property of $MU$ can then be rephrased: the spectrum $MU$ is
universal among $E_\infty$ ring spectra $E$ equipped with a coherently
commutative diagram
\[
\xymatrix{
ku \ar[r] \ar[d] &
\pic(\mb S) \ar[d] \\
H\mb Z \ar@{.>}[r] &
\pic(E).
}
\]
(More concretely, this asks for a map $H\mb Z \to \pic(E)$ and a chosen
homotopy between the two composites.)

This allows us to exploit Arone--Lesh's sequence of spectra
interpolating between $ku$ and $H\mb Z$, giving us an inductive
sequence of obstructions to $E_\infty$ orientations.
\begin{thm}
There exists a filtration of $MU$ by $E_\infty$ Thom spectra
\[
\mb S\to MX_1 \to MX_2 \to MX_3 \to \cdots \to MU
\]
with the following properties.
\begin{enumerate}
\item The map $\hocolim MX_i \to MU$ is an equivalence.
\item There is a canonical complex orientation of $MX_1$ such that, for
  all $E_\infty$ ring spectra $E$, the space $\Map_{E_\infty}(MX_1,E)$
  is homotopy equivalent to the space of ordinary complex orientations
  of $E$.
\item For all $m > 0$ and all maps of $E_\infty$ ring spectra $MX_{m-1}
  \to E$, the space of extensions to a map of $E_\infty$ ring spectra
  $MX_m \to E$ is a homotopy pullback diagram of the form
  \[
  \xymatrix{
    \Map_{E_\infty} (MX_m, E) \ar[r] \ar[d] &
    \Map_{E_\infty} (MX_{m-1}, E) \ar[d] \\
    \{\ast\} \ar[r] & \Map_*(F_m, \Pic(E))
  }
  \]
  for a certain fixed space $F_m$, where $\Pic(E)$ is the classifying
  space of the category of smash-invertible $E$-modules.

  More specifically, given an $E_\infty$ map $MX_{m-1} \to E$, there
  is an $E$-module Thom spectrum $M_E\xi$ classified by a map $\xi\co
  F_m \to \Pic(E)$. An extension to an $E_\infty$ map $MX_m \to E$
  exists if and only if there is an orientation $M_E \xi \to E \wedge
  S^{2m}$ in the sense of \cite{ando-blumberg-gepner-hopkins-rezk},
  and the space of extensions is naturally equivalent to the space
  $\Or(\xi)$ of orientations.
\item The map $MX_{m-1} \to MX_m$ is a rational equivalence if $m >
  1$, a $p$-local equivalence if $m$ is not a power of $p$, and a
  $K(n)$-local equivalence if $m > p^n$.
\end{enumerate}
\end{thm}
In particular, the spectrum $MX_1$ will be a universal complex
oriented $E_\infty$ ring spectrum described by Baker--Richter
\cite{baker-richter-thom}.

The suspended spaces $F_m$ are explicitly described by
\cite{arone-lesh-filtered} as being derived orbit spectra
$(L_m)^\diamond \wedge^{\mb L}_{U(m)} S^{2m}$. Here $L_m$ is the nerve
of the (topologized) poset of proper direct-sum decompositions of $\mb
C^m$, $S^{2m}$ is the one-point compactification of $\mb C^m$, and
$\diamond$ denotes unreduced suspension. Alternatively, the space
$F_m$ can be described as the homotopy cofiber of the map of Thom
spaces
\[
(L_m \times_{U(m)} EU(m))^{\gamma_m} \to BU(m)^{\gamma_m}
\]
for the universal bundle $\gamma_m$.

In the particular case of an $E(1)$-local $E_\infty$ ring spectrum
$E$, such as a form of $K$-theory \cite[Appendix A]{tmforientation},
this will allows us to verify that Ando's criterion is both necessary
and sufficient if $E^*$ is torsion-free.  At higher chromatic levels
there are expected to be secondary and higher obstructions involving
relations between power operations.

\begin{rmk}
  Rognes had previously constructed a similar filtration on algebraic
  $K$-theory spectra \cite{rognes-rankfiltration}, further examined in
  the case of complex $K$-theory in
  \cite{arone-lesh-rankfiltration}. This filtration gives rise to a
  sequence of spectra interpolating the map $* \to ku$ rather than $ku
  \to \eilm{\mb Z}$. On taking Thom spectra of the resulting infinite
  loop maps to $\mb Z \times BU$, the result should be a construction
  of the periodic complex bordism spectrum $MUP$ with very similar
  properties but slightly different subquotients, relevant to
  a more rigid orientation theory for $2$-periodic spectra.
\end{rmk}

The authors would like to thank Greg Arone and Kathryn Lesh for
discussions related to this material, to Andrew Baker, Eric Peterson,
and Nathaniel Stapleton for their comments, and to Jeremy Hahn for
locating an error in the previous version of this paper.

\section{The filtration of connective $K$-theory}

In this section, we will give short background on the results that we
require from \cite[3.9, 8.3, 9.4, 9.6, 11.3]{arone-lesh-filtered}.

\begin{prop}
  \label{prop:arone-lesh-space}
  There exists a sequence of maps of $E_\infty$ spaces
  \[
    B_0 \to B_1 \to B_2 \to B_3 \to \cdots
  \]
  with the following properties.
  \begin{enumerate}
  \item The space $B_\infty = \hocolim B_m$ is equivalent to the
    discrete $E_\infty$ space $\mb N$, and the induced maps $\pi_0 B_m
    \to \mb N$ are isomorphisms.
  \item The space $B_0$ is the nerve $\coprod BU(n)$ of a skeleton of
    the category of finite-dimensional vector spaces and isomorphisms,
    with the $E_\infty$ structure induced by direct sum.
  \item Let $\mb P$ denote the functor taking a space $X$ to the free
    $E_\infty$ space on $X$ \cite[3.5]{may-loopspaces}, with the
    homotopy type
    \[
    \mb P(X) = \coprod_{n \geq 0} (X^n)_{h\Sigma_n}.
    \]
    For each $m > 0$, there is a homotopy
    pushout diagram of $E_\infty$ spaces
    \[
      \xymatrix{
        \mb P({F_m}) \ar[r] \ar[d] & B_{m-1} \ar[d] \\
        \mb P(*) \ar[r] & B_m,
      }
    \]
    where $F_m$ is the path component of $B_{m-1}$ mapping to $m
    \in \mb N$.
  \item The map $F_m \to *$ is an isomorphism in rational homology
    if $m > 1$, an isomorphism in $p$-local homology if $m$ is not a
    power of $p$, and an isomorphism in $K(n)$-homology if $m > p^n$.
  \item The spectrum $\Sigma^\infty F_m$ is $(2m-1)$-connnected.
  \end{enumerate}
\end{prop}

The explicit description allows analysis of the filtration quotients
using \cite[2.5]{arone-lesh-filtered}.  The space $F_1$ is the path
component $BU(1) \subset \coprod BU(n)$.  We then get a homotopy
commutative diagram
\begin{equation}
  \label{eq:powerconstruction}
  \xymatrix{
    B(\Sigma_p \wr U(1)) \ar[rr] \ar[dr] \ar[dd] && BU(p) \ar[d] \\
    & \mb P(BU(1)) \ar[r] \ar[d] & \coprod BU(n) \ar[d] \\
    B\Sigma_p \ar[r] & \mb P(*) \ar[r] & B_1
  }
\end{equation}
with the top map induced by the inclusion of the monomial matrices
in $U(p)$ and the left map induced by the projection $\Sigma_p \wr
U(1) \to \Sigma_p$.  The
homotopy pushout of the subdiagram
\begin{equation}
  \label{eq:pthspace}
\xymatrix{
  B\Sigma_p & \ar[l] B(\Sigma_p \wr U(1)) \ar[r] &
  BU(p)
}
\end{equation}
maps to $F_p$ by a $p$-local homotopy equivalence.

Arone--Lesh then apply an infinite loop space machine to the sequence
of Proposition~\ref{prop:arone-lesh-space}, with the following result.

\begin{cor}
  There exists a sequence of connective spectra
  \begin{equation}
    \label{eq:al-filtration}
    b_0 \to b_1 \to b_2 \to b_3 \to \cdots
  \end{equation}
  with the following properties.
  \begin{enumerate}
  \item The spectrum $b_\infty = \hocolim b_m$ is equivalent to the
    spectrum $H\mb Z$, and the induced maps $\pi_0 b_n \to \mb Z$ are
    isomorphisms.
  \item The spectrum $b_0$ is the complex $K$-theory spectrum $ku$.
  \item For each $m > 0$, there is a homotopy pushout diagram
    \begin{equation}
      \label{eq:spectrum-pushout}
      \xymatrix{
        \Sigma^\infty (F_m)_+ \ar[r] \ar[d] & b_{m-1} \ar[d] \\
        \Sigma^\infty S^0 \ar[r] & b_m.
      }
    \end{equation}
  \item The map $b_{m-1} \to b_m$ is an isomorphism in rational
    homology if $m > 1$, an isomorphism in $p$-local homology if $m$
    is not a power of $p$, and an isomorphism in $K(n)$-homology if $m
    > p^n$.
  \item The homotopy fiber $\Sigma^\infty F_m$ of the map $b_{m-1} \to
    b_m$ is $(2m-1)$-connected.
  \end{enumerate}
\end{cor}

\section{The filtration of $BU$}

\begin{defn}
For each $m \geq 0$, let $x_m$ be the homotopy fiber $\hofib(ku
\to b_m)$ of the maps from equation~(\ref{eq:al-filtration}), and let
$X_m = \Omega^\infty x_m$.
\end{defn}
This gives rise to a sequence of maps
\begin{equation}
  \label{eq:spectrum-filtration}
  * \simeq x_0 \to x_1 \to x_2 \to x_3 \to \cdots,
\end{equation}
with homotopy colimit $bu \simeq \Sigma^2 ku$ by Bott periodicity.
For each $m > 0$, diagram~(\ref{eq:spectrum-pushout}) and the
octahedral axiom imply that the homotopy fiber of $x_{m-1} \to x_m$ is
equivalent to the desuspension $\Omega \Sigma^\infty F_m$ of the
reduced suspension spectrum.

Applying $\Omega^\infty$ to the filtration of
equation~(\ref{eq:spectrum-filtration}), we obtain a filtration of
$BU$ by infinite loop spaces:
\begin{equation}
  \label{eq:space-filtration}
  * \to X_1 \to X_2\to X_3 \to \cdots
\end{equation}
The homotopy fiber of $X_{m-1} \to X_m$ is the space $\Omega QF_m =
\Omega^{\infty+1} \Sigma^\infty F_m$.

In order to analyze the effect of these maps in $K(n)$-local homology,
we will require some preliminary results.

\begin{lem}
  \label{lem:princfib}
  Suppose $x \to y \to z$ is a fiber sequence of spectra,
  $\Omega^\infty x$ is $K(n)$-locally trivial, and $x$ is
  connective. Then the map $\Omega^\infty y \to \Omega^\infty z$
  induces an isomorphism on $K(n)$-homology.
\end{lem}

\begin{proof}
  By assumption $\pi_0 y \to \pi_0 z$ is surjective, so the map
  $\Omega^\infty y \to \Omega^\infty z$ is a principal fibration whose
  fiber over any point is $\Omega^\infty x$. Applying the (natural)
  generalized Atiyah--Hirzebruch spectral sequence, we obtain a
  spectral sequence
  \[
    \mathcal{H}_*(\Omega^\infty z; K(n)_* (\Omega^\infty x)) \Rightarrow K(n)_*
    (\Omega^\infty y),
  \]
  where the $E_2$-term may be homology with coefficients in a local
  coefficient system. By assumption, the edge morphism to the
  Atiyah--Hirzebruch spectral sequence
  \[
    H_*(\Omega^\infty z; K(n)_*) \Rightarrow K(n)_*(\Omega^\infty z)
  \]
  is an isomorphism on $E_2$-terms, so it converges to an isomorphism
  $K(n)_* (\Omega^\infty y) \to K(n)_* (\Omega^\infty z)$.
\end{proof}

\begin{prop}
  \label{prop:connectivitypres}
  Let $W$ be a based space whose suspension spectrum is at least
  $k$-connected, and define $\mathcal{N}$ be the family of spectra $T$
  such that $\Omega^\infty(T \smsh{} W)$ is $K(n)$-acyclic.  The
  family $\mathcal{N}$ has the following properties:
  \begin{enumerate}
  \item $\mathcal{N}$ is closed under finite wedges.
  \item $\mathcal{N}$ is closed under filtered homotopy colimits.
  \item Suppose $T' \to T \to T''$ is a fiber sequence such that $T'$
    is a $(-k-1)$-connected spectrum in $\mathcal{N}$. Then $T$ is in
    $\mathcal{N}$ if and only if $T''$ is in $\mathcal{N}$.
  \item If $\widetilde H_{k+1}W$ is torsion, $\mathcal{N}$ contains
    the Eilenberg--Mac Lane spectrum $\Sigma^{n-k} HA$ for any abelian
    group $A$.
  \item If $W$ is $K(n)$-locally trivial, $\mathcal{N}$ contains
    $S^0$.
  \item If $W$ is $K(n)$-locally trivial and $\widetilde H_{k+1} W$ is
    torsion, $\mathcal{N}$ contains all $(n-k-1)$-connected
    spectra.
  \end{enumerate}
\end{prop}

\begin{rmk}
  In particular, since $\Sigma^\infty W$ is $k$-connected the
  assumption on $H_{k+1} W$ holds automatically with $k$ replaced by
  $(k-1)$. Therefore, $\Sigma^{n-k+1} HA$ is in $\mathcal{N}$ for any
  $A$, and if $W$ is $K(n)$-acyclic all $(n-k)$-connected spectra are
  in $\mathcal{N}$.
\end{rmk}

\begin{proof}
  We will prove these items individually.
  \begin{enumerate}
  \item There is a weak equivalence
    \[
      \Omega^\infty(\vee_{i=1}^N T_i \smsh{} W)
      \to \prod_{i=1}^N \Omega^\infty(T_i \smsh{} W),
    \]
    and so this follows from Morava $K$-theory's K\"unneth formula
    \[
      K(n)_*(X \times Y) \cong K(n)_* X \tens{K(n)_*} K(n)_* Y.
    \]
  \item The functor $\Omega^\infty(T \smsh{} W)$ preserves filtered
    homotopy colimits in $T$, and so there is an isomorphism
    \[
      \varinjlim K(n)_* (\Omega^\infty (T_\alpha \smsh{} W)) \cong 
      K(n)_* (\Omega^\infty ((\hocolim T_\alpha) \smsh{} W)).
    \]
    Therefore, $\hocolim T_\alpha$ is in $\mathcal{N}$ if the
    $T_\alpha$ are.
  \item The spectrum $T' \smsh{} W$ is connective, and so the result
    follows from a direct application of Lemma~\ref{lem:princfib}.
  \item The spectrum $\Sigma^{n-k} H\mb Z \smsh{} W$ is an 
    $n$-connected generalized Eilenberg--Mac Lane spectrum, and so
    there is a weak equivalance
    \[
      \Omega^\infty(\Sigma^{n-k} H\mb Z \smsh{} W) \xrightarrow{\sim}
      \prod_{i=n+1}^\infty K(\widetilde H_{k-n+i} W, i).
    \]
    By the work of Ravenel--Wilson
    \cite{ravenel-wilson-eilenbergmaclane}, $K(n)_* K(A,i)$ is trivial
    if $i > n+1$ or if $i = n+1$ and $A$ is a torsion abelian group,
    and so by the K\"unneth formula $\Sigma^{n-k} H\mb Z$ is in
    $\mathcal{N}$.

    Applying item 1, we find $\Sigma^{n-k} H(\mb Z^N)$ is in
    $\mathcal{N}$; applying item 2, we find $\Sigma^{n-k} HF$ is in
    $\mathcal{N}$ whenever $F$ is free abelian; applying item 3 to the
    fiber sequence
    \[
      \Sigma^{n-k} HR \to \Sigma^{n-k} HF \to \Sigma^{n-k} HA
    \]
    associated to a free resolution $0 \to R \to F \to A \to 0$, we
    find that $\Sigma^{n-k} HA$ is in $\mathcal{N}$.
  \item The Snaith splitting \cite{snaith-stabledecomposition} shows
    that we have a decomposition
    \[
      \Sigma^\infty(QW)_+ \simeq \bigvee_{k \geq 0} (\Sigma^\infty
      W^{\smsh{} k})_{h\Sigma_k}.
    \]
    The spectra $\Sigma^\infty W^{\smsh{} k} \to *$ are $K(n)$-locally
    trivial for $k > 0$, and so the same is true of the homotopy orbit
    spectra.

    (Note that Snaith splitting is required to deduce this equivalence
    on the level of Morava $K$-theory because the equivalence does not
    hold on the level of mod-$p$ homology. In particular, the
    $K(n)$-homology of a smash-power of $Z$ is not a functor of the
    $K(n)$-homology of $Z$.)

  \item First suppose that $T$ is connective. By items 5 and 3,
    $\mathcal{N}$ contains any sphere $S^i$ for $i \geq 0$. By items 1
    and 2, $\mathcal{N}$ contains any wedge $\vee S^i$ for $i \geq
    0$. By item 3, induction on the dimension shows that $\mathcal{N}$
    contains any connective finite-dimensional CW-spectrum. By item 2,
    $\mathcal{N}$ then contains any connective spectrum.
    
    We can now prove the general case. Suppose $T$ is
    $(n-k-1)$-connected with $n-k < 0$, and consider the following
    portion of the Whitehead tower of $T$:
    \[
      \xymatrix{
        T[0,\infty) \ar[r] & T[-1,\infty) \ar[r] \ar[d] & \dots \ar[r]
        & T[n-k+1,\infty) \ar[r] \ar[d] & T \ar[d]\\
        & \Sigma^{-1} H\pi_{-1} T & & \Sigma^{n-k+1} H\pi_{n-k+1} T &
        \Sigma^{n-k} H\pi_{n-k} T
      }
    \]
    We have just shown that $T[0,\infty)$ is in $\mathcal{N}$ because
    it is connective, and the spectra $\Sigma^i H\pi_i T$ are in
    $\mathcal{N}$ for $n-k \leq i \leq -1$ by item 4. By inductively
    applying item 3 we find that $T$ is in $\mathcal{N}$.\qedhere
  \end{enumerate}
\end{proof}

\begin{prop}
  \label{prop:trivialfibers}
  The natural map $\Omega QF_m \to *$ is a rational homology equivalence for
  $m > 1$, a $p$-local equivalence for $m$ not a power of $p$, and a
  $K(n)$-local equivalence if $m > p^n$.
\end{prop}

\begin{proof}
  When $m = 1$ there is nothing to show. When $m>1$ the suspension
  spectrum of $F_m$ is $k$-connected for some $k\geq 2m-1$. The space
  $F_m$ is also rationally trivial and $p$-locally trivial unless $m$
  is a power of $p$, so the rational homology and homotopy groups are
  always torsion and have $p$-torsion only when $m$ is a power of $p$;
  hence the same is true for both $\Sigma^\infty F_m$ and $\Omega Q
  F_m$. If $m > p^n$ then
  \[
    n-k-1 \leq n - (2m-1) - 1 < n - 2p^n \leq 1 - 2p < -2,
  \]
  so $S^{-1}$ is $(n-k-1)$-connected. By
  Proposition~\ref{prop:connectivitypres} item 6, we then find that
  $\Omega QF_m = \Omega^{\infty}(S^{-1} \smsh{} F_m)$ is
  $K(n)$-locally trivial.
\end{proof}

\section{Decomposition of $MU$}

\begin{defn}
For each $m \geq 0$, let $MX_m$ be the Thom spectrum of the infinite
loop map $X_m \to BU$.
\end{defn}
From the sequence~(\ref{eq:space-filtration}) of infinite loop spaces
over $BU$, we obtain a filtration of $MU$ by $E_\infty$ ring spectra:
\begin{equation}
  \label{eq:thom-filtration}
  \mb S \to MX_1 \to MX_2 \to MX_3 \to \cdots
\end{equation}

\begin{prop}
  For any associative $MU$-algebra $E$ such that $\Omega QF_m \to *$
  is an $E_*$-isomorphism, the map $MX_{m-1} \to MX_m$ induces an
  isomorphism in $E$-homology.
\end{prop}

\begin{proof}
  After smashing with $MU$, the Thom diagonal makes the sequence of
  equation~(\ref{eq:thom-filtration}) equivalent to the sequence of
  $MU$-algebras
  \[
  MU \to MU[X_1] \to MU [X_2] \to
  \cdots \to MU[BU],
  \]
  where for an $E_\infty$ space $M$ we define $MU[M]$ to be
  $E_\infty$ ring spectrum $MU \smsh{} M_+$.

  The fiber sequence $\Omega QF_m \to X_{m-1} \to X_m$ of $E_\infty$
  spaces implies that there are equivalences
  \[
  MU \smsh{MU[\Omega QF_m]} (MU \wedge MX_{m-1}) \simeq MU \wedge
  MX_m.
  \]
  Smashing this identification over $MU$ with $E$ translates this into
  an identity
  \[
    E \smsh{E[\Omega QF_m]} (E\smsh{} MX_{m-1}) \simeq E \smsh{}
    MX_m.
  \]
  Since the natural map $E[\Omega QF_m] \to E$ is an equivalence by
  assumption, the result follows.
\end{proof}

By work of Lazarev \cite{lazarev-ainfty}, $K(n)$ admits the
structure of an associative $MU$-algebra and so we can specialize this
result to the case where $E$ is a Morava $K$-theory. Combined with
Proposition~\ref{prop:trivialfibers}, this gives the following
result.
\begin{cor}
  The map $MX_{m-1} \to MX_m$ is a rational equivalence for $m > 1$, a
  $p$-local equivalence for $m$ not a power of $p$, and a $K(n)$-local
  equivalence if $m > p^n$.
\end{cor}

In particular, we have the following equivalences:
\begin{align*}
  (MX_1)_{\mb Q} &\simeq MU_{\mb Q}\\
  L_{K(n)} MX_{p^n} &\simeq L_{K(n)} MU\\
  L_{E(n)} MX_{p^n} &\simeq L_{E(n)} MU
\end{align*}

\begin{rmk}
  This filtration on $MU$ relies only on the existence of the
  map $\coprod BU(n) \to \mb N$ of $E_\infty$ spaces. In particular,
  this construction is naturally equivariant for the action of the
  cyclic group $C_2$, determines an equivariant filtration of the
  Real $K$-theory spectrum, and a sequence of $C_2$-equivariant
  $E_\infty$ Thom spectra filtering the Real bordism spectrum $MU_{\mb
    R}$. However, the $K(n)$-local properties of this filtration
  appear to be less straightforward.
\end{rmk}

\section{Picard groups}

As described in the introduction, for an $E_\infty$ ring spectrum $E$
we let $\Pic(E)$ be the nerve of the symmetric monoidal category of
smash-invertible $E$-modules and weak equivalences
\cite{hopkins-mahowald-sadofsky,mathew-stojanoska-pictmf}. The
symmetric monoidal structure makes $\Pic(E)$ into a grouplike
$E_\infty$ space, and we write $\pic(E)$ for the associated
spectrum. The $0$-connected cover of $\pic(E)$ is $bgl_1(E)$.

As in \cite{ando-blumberg-gepner-hopkins-rezk}, a map $\xi\co X \to
\Pic(E)$ over the path component of an invertible $E$-module $E^\zeta$
parametrizes families of $E$-modules over $X$ with fibers equivalent
to $E^\zeta$, and there is an associated $E$-module Thom spectrum
$M_E\xi$. (Technically, to apply the results of
\cite{ando-blumberg-gepner-hopkins-rezk} we first need to smash with
the element $E^{-\zeta} \in \pi_0 \Pic(E)$ to move the target to
$BGL_1(E)$.)

The functor sending a complex vector space to the suspension spectrum
of its one-point compactification gives a map of $E_\infty$ spaces
$\coprod BU(m) \to \Pic(\mb S)$, and the associated map of spectra is
a map $ku \to \pic(\mb S)$.

The space $\Map_{E_\infty}(MX_m, E)$ is naturally equivalent to the
space of nullhomotopies of the composite map $x_m \to bu \to bgl_1(E)$
\cite{ando-hopkins-rezk-kotheory, ando-blumberg-gepner-hopkins-rezk}.
However, the spectra $x_m$ are $0$-connected, so this is equivalent to
the space of extensions in the diagram
\[
\xymatrix{
  ku \ar[d] \ar[r] & \pic(\mb S)\ar[d]  \\
  b_m \ar@{.>}[r] & \pic(E).
}
\]
If we have already fixed an extension $b_{m-1} \to \pic(E)$, the
pushout diagram~(\ref{eq:spectrum-pushout}) expresses the space of
compatible extensions to $b_m$ as the space of commutative diagrams
\begin{equation}
  \label{eq:pushout}
  \xymatrix{
    F_m \ar[r] \ar[d] & B_{m-1}\ar[d]\\
    CF_m \ar@{.>}[r] & \Pic(E).
  }
\end{equation}
We write $\xi$ for the diagonal composite $F_m \to \Pic(E)$ in this
diagram.

\begin{prop}
Given an extension of $ku \to \pic(E)$ to a map $b_{m-1} \to \pic(E)$,
the space of extensions to a map $b_m \to \pic(E)$ is equivalent to the
space $\Or(\xi)$ of orientations of the $E$-module Thom spectrum
$M_E\xi$ over $F_m$.
\end{prop}

\begin{proof}
We must show that the space of homotopies from $\xi$ to a constant map
is equivalent to the space $\Or(\xi)$ of orientations: maps of
$E$-modules $M_E\xi \to E \wedge S^{2m}$ which restrict to an equivalence on
Thom spectra at each point.

In our case, we may choose a basepoint $* \in BU(m)$ classifying the
vector bundle $\mb C^m \to *$, whose image $* \to F_m \to \Pic(E)$
corresponds to the $E$-module $E \wedge S^{2m}$. We construct the
following diagram of pullback squares.
\[
\xymatrix{
  \mathop{\Or}(\xi) \ar[r] \ar[d] & \{\xi\} \ar[d] \\
  \Map_*(CF_m,\Pic(E)) \ar[r] \ar[d] &
  \Map_*(F_m, \Pic(E)) \ar[r] \ar[d] & \{E \wedge S^{2m}\} \ar[d] \\
  \Map(CF_m,\Pic(E)) \ar[r] & \Map(F_m, \Pic(E)) \ar[r] & \Map(*,\Pic(E))
}
\]
Here the bottom map, of necessity, lands in the path component of the
$E$-module $E \wedge S^{2m}$, and the upper-left pullback is the space
$\Or(\xi)$ because $F_m$ is connected. Therefore, given this map
$\xi\co F_m \to \Pic(E)$, the space of extensions is equivalent to the
space of orientations of the $E$-module Thom spectrum $M_E\xi$.
\end{proof}

\section{Orientation towers}
\label{sec:orientation-towers}

The space of $E_\infty$ orientations $MU \to E$ can now be expressed
as the homotopy limit of the tower
\[
  \cdots \to \Map_{E_\infty}(MX_3,E) \to \Map_{E_\infty} (MX_2, E) \to
  \Map_{E_\infty} (MX_1, E) \to *.
\]
The description of the space of extension diagrams from
equation~(\ref{eq:pushout}) is equivalent to a homotopy pullback
square
\[
\xymatrix{
  \Map_{E_\infty}(MX_m,E) \ar[r] \ar[d] &
  \Map_{E_\infty}(MX_{m-1},E) \ar[d] \\
  \{*\} \ar[r]^-{E \wedge S^{2m}} &
  \Map_*(F_m, \Pic(E)),
}
\]
where the bottom arrow classifies a constant map to the component of
$E \wedge S^{2m}$ in $\Pic(E)$.  The space of lifts is the space
of orientations of the Thom spectrum on $F_m$, and so the unique
obstruction to a lifting is the existence of a Thom class.

When $m=1$, the space $\Map_{E_\infty}(MX_1,E)$ is the space of
orientations of the Thom spectrum classified by the composite
\[
BU(1) \to \coprod BU(n) \to \Pic(E).
\]
More specifically, the Thom spectrum of this composite is $E \wedge
MU(1) \simeq E \wedge BU(1)$. Orientations of this are classical
complex orientations: the space of orientations of this Thom spectrum
is the space of maps of $E$-modules $c_1\co E \wedge BU(1) \to E
\wedge S^2$ which restrict to the identity map of $E \wedge S^2$.
Therefore, $\Map_{E_\infty}(MX_1,E)$ is naturally the space of
ordinary complex orientations of $E$.

\section{Symmetric power operations}

In order to study $p$-local orientations by $MX_p$, we
will need to recall the construction of power operations.

Associated to a complex vector bundle $\xi \to X$, we have the Thom
space $\Th(\xi)$. This is functorial in maps of vector bundles which are
fiberwise injections, and for the exterior Whitney sum $\boxplus$
there is a natural isomorphism
\[
\Th(\xi \boxplus \xi') \cong \Th(\xi) \wedge \Th(\xi')
\]
that is part of a strong monoidal structure on $\Th$.

\begin{defn}
We define the following symmetric power functors:
\begin{align*}
  \Pm^\times(X) &= (X^{\times m})_{h\Sigma_m}&&\text{for $X$ a space.}\\
  \Pm^\wedge(X) &= (X^{\wedge m})_{h\Sigma_m}&&\text{for $X$ a based space.}\\
  \Pm^{\smsh{E}}(X) &= (X^{\wedge_E m})_{h\Sigma_m} &&\text{for $X$ an
    $E$-module.}
\end{align*}
For any of the symmetric monoidal structures $\owedge$ above, we will
write
\[
\Dm^\owedge\co X^{\owedge m} \to \Pm^\owedge(X)
\]
and
\[
\Delta^\owedge_m\co \Pm^\owedge(X \owedge Y) \to \Pm^\owedge(X)
\owedge \Pm^\owedge(Y)
\]
for the associated natural transformations.
\end{defn}

For a vector bundle $\xi \to X$, there is a natural vector bundle
structure on the map $\Pm^\times(\xi) \to \Pm^\times(X)$, and we have
a pullback diagram of vector bundles
\[
\xymatrix{
\xi^{\boxplus m} \ar[r]^-{\Dm^\times} \ar[d] & \Pm^\times(\xi) \ar[d]\\
X^{\times m} \ar[r]_-{\Dm^\times} & \Pm^\times(X).
}
\]
\begin{prop}
There are natural isomorphisms:
\begin{align*}
\Th(\Pm^\times \xi) &\cong \Pm^\wedge \Th(\xi)\\
E \wedge \Th(\xi) &\cong M_E(\xi)\\
E \wedge \Pm^\wedge(Y) &\cong \Pm^{\wedge_E}(E \wedge Y)
\end{align*}
\end{prop}

\begin{defn}
Suppose $E$ has a chosen complex orientation $u$, and let $\varepsilon$
be the trivial complex vector bundle over a point. We write
\[
t_u(\xi)\co M_E\xi \to M_E(\varepsilon^{\dim(\xi)})
\]
for the natural $E$-module complex orientation.
\end{defn}
The map $t_u(\varepsilon)$ is the identity map of $M_E(\varepsilon)
\cong E \wedge S^2$, and $t_u(\gamma_1) = c_1$ for the tautological
bundle $\gamma_1 \to BU(1)$. Orientations commute with exterior sum:
the strong monoidal structure of $\Th$ gives us an identification
\[
t_u(\xi \boxplus \xi') \cong t_u(\xi) \wedge_E t_u(\xi').
\]
Naturality of these orientations in pullback diagrams holds:
for any map $f\co X \to Y$ and vector bundle $\xi \to Y$ we have 
\[
t_u(f^* \xi) = t_u(\xi) \circ M_E(f).
\]
In particular, this implies that $t_u(\xi^{\boxplus m}) =
t_u(\Pm^\times \xi) \circ M_E(\Dm^\times)$.

\begin{defn}
Write $\rho_m$ for the vector bundle $\Pm^\times(\varepsilon)$ on
$B\Sigma_m$, associated to the permutation representation of $\Sigma_m$ on
$\mb C^m$.
\end{defn}
For this vector bundle, the naturality of $\Dm^{\smsh{E}}$
implies that the Thom class $t_u(\rho_m^{\boxplus
  k}) \circ \Dm^{\smsh{E}}$ is the identity map.

The identities above allow us to verify several relations between Thom
classes.
\begin{prop}
For a complex vector bundle $\xi \to X$ and a point $i\co * \to X$
with a chosen lift to $i\co \varepsilon \to \xi$, we have the
following.
\begin{align*}
  t_u(\xi^{\boxplus m}) &= t_u(\Pm^\times(\xi)) \circ \Dm^{\smsh{E}}\\
  t_u(\xi^{\boxplus m}) &= t_u(\rho_m^{\oplus \dim(\xi)}) \circ \Pm^{\wedge_E}(t_u(\xi)) \circ
  \Dm^{\smsh{E}} \\
  t_u(\rho_m^{\oplus \dim(\xi)}) &=  t_u(\Pm^\times(\xi)) \circ \Pm^{\smsh{E}}(M_Ei) \\
  t_u(\rho_m^{\oplus \dim(\xi)}) &=  t_u(\rho_m^{\oplus \dim(\xi)}) \circ \Pm^{\smsh{E}}(t_u(\xi)) \circ \Pm^{\smsh{E}}(M_Ei)
\end{align*}
\end{prop}
\begin{cor}
\label{cor:orientationrestriction}
For a complex vector bundle $\xi \to X$, the map $t_u(\rho_m) \circ
\Pm^{\wedge_E}(t_u(\xi))$ is an orientation of the Thom spectrum
$M_E(\Pm^\times(\xi))$ over $\Pm^\times X$ which coincides with
$t_u(\Pm^\times( \xi))$ after restriction to $X^{\times m}$ or
$B\Sigma_m$.
\end{cor}

\section{Power operations}

Assume that $p$ is a fixed prime and $E$ is an $E_\infty$ ring
spectrum with a chosen complex orientation $u$. In this section we
will recall power operations on even-degree cohomology classes
\cite{rezk-power-operations}; in the case $E = MU$ these were
constructed by tom Dieck and Quillen, and used by Ando in his
characterization of $H_\infty$ structures
\cite{ando-ellipticpowerops}.

From here on, we will write $\rho = \rho_p$ for the permutation
representation of $\Sigma_p$.
\begin{defn}
For an $E$-module spectrum $M$, we define
\[
\mathcal{P}_u\co [M,E \wedge S^{2k}]_E \to [\Pp^{\smsh{E}} M, E \wedge
S^{2pk}]_E
\]
by the formula $\mathcal{P}_u(\alpha) = t_u(\rho^{\oplus k}) \circ
\Pp^{\smsh{E}}(\alpha)$.
\end{defn}
These power operations satisfy a multiplication formula. For
$E$-module spectra $M$ and $N$ with maps $\alpha \in [M,E \wedge
S^{2k}]$ and $\beta \in [N, E \wedge S^{2l}]$, we can form
\[
\alpha \wedge_E \beta \in [M \wedge_E N, E \wedge S^{2(k+l)}].
\]
Then there is a natural identity
\[
\mathcal{P}_u(\alpha \wedge_E \beta) \circ \Delta^{\wedge_E}_p=
\mathcal{P}_u(\alpha) \wedge_E \mathcal{P}_u(\beta).
\]
These power operations also depend on $u$, except when $k = 0$ where
they agree with the ordinary extended power construction.

\begin{rmk}
While the formula for the power operations is given using even
spheres, it implicitly relies on a fixed identification of $S^{2k}$
with the one-point compactification of a complex vector space.
\end{rmk}

\begin{defn}
For a complex vector bundle $\xi \to X$, let $j\co B\Sigma_p \times X
\to \Pp^\times(X)$ be the diagonal, and let $\rho \boxtimes \xi \to
B\Sigma_p \times X$ be the external tensor vector bundle $j^*
\Pp^\times(\xi)$. We define
\[
P_u\co E^{2k}(\Th(\xi)) \to 
E^{2pk}(\Th(\rho \boxtimes \xi))
\]
by the formula $P_u(\alpha) = \mathcal{P}_u(\alpha) \circ
M_E(j)$.
\end{defn}

\begin{defn}
For a $p$-local, complex orientable multiplicative cohomology theory
$F$, the transfer ideal $I_{tr} \subset F^*(B\Sigma_p)$ is the image
of the transfer map $F^* \to F^*(B\Sigma_p)$, generated by the image
of $1$ under the transfer. For any $Y$ equipped with a chosen map
to $B\Sigma_p$, we also write $I_{tr}$ for the ideal of $F^*(Y)$
generated by the image of $I_{tr}$.
\end{defn}

The natural transformations $P_u$ are multiplicative but not additive,
instead satisfying a Cartan formula. The terms in the Cartan formula
which obstruct additivity are transfers from the cohomology of proper
subgroups of the form $\Sigma_k \times \Sigma_{p-k} \subset \Sigma_p$.
If $E$ is $p$-local, in the evenly-graded ring $E^{2*}(\Th(\rho
\boxtimes \xi))$ the mixed terms in the Cartan formula are
contained inside the transfer ideal $I_{tr} \cdot E^{2*}(\Th(\rho
\boxtimes \xi))$.

\begin{prop}
The maps $P_u$ reduce to natural maps
\[
\Psi_u\co E^{2*}(\Th(\xi)) \to E^{2p*}(\Th(\rho \boxtimes \xi)) / I_{tr}
\]
that are additive and take any Thom class for $\xi$ to a Thom class
for $\rho \boxtimes \xi$. The maps $\Psi_u$ are multiplicative, in the
sense that for elements $\alpha \in E^{2*}(\Th(\xi))$ and $\beta \in
E^{2*}(\Th(\xi'))$ we have $\Psi_u(\alpha \beta) = \Psi_u(\alpha)
\Psi_u(\beta)$.
\end{prop}

\section{Cohomology calculations}

In this section, we fix a $p$-local, complex orientable multiplicative
cohomology theory $F$. Choosing a complex orientation of $F$, we use
$\mb G$ to denote the associated formal group law over $F^*$, and
$[n]_{\mb G}(x)$ the power series representing the associated $n$-fold
sum $(x +_{\mb G} x +_{\mb G} \cdots +_{\mb G}x)$.

\begin{prop}
\label{prop:cohomologyrestriction}
The restriction map
\[
F^*(B (\Sigma_p \wr U(1))) \to F^*(BU(1)^p) \times F^*(B\Sigma_p
\times BU(1))
\]
is injective.
\end{prop}

We will discuss a proof of this result that requires more
multiplicative structure from $E$ but applies to a wider variety of
objects than $BU(1)$ in Section~\ref{sec:extra}.

\begin{proof}
Writing $B(\Sigma_p \times U(1)) \to B(\Sigma_p \wr U(1))$ as a map of
homotopy orbit spaces $BU(1)_{h\Sigma_p} \to (BU(1)^p)_{h\Sigma_p}$,
we obtain a diagram of function spectra
\[
\xymatrix{
F(B(\Sigma_p \wr U(1)), F) \ar[r]^-\sim \ar[d] &
F(BU(1)^p, F)^{h\Sigma_p} \ar[d]\\
F(B(\Sigma_p \times U(1)), F) \ar[r]_-\sim &
F(BU(1),F)^{h\Sigma_p}.
}
\]
Therefore, the map on cohomology is the abutment of a map of homotopy
fixed-point spectral sequences:
\[
\xymatrix{
H^s(\Sigma_p, F^t(BU(1)^p)) \ar@{=>}[r] \ar[d] &
F^{t+s}(B(\Sigma_p \wr U(1))) \ar[d] \\
H^s(\Sigma_p, F^t(BU(1))) \ar@{=>}[r] &
F^{t+s}(B(\Sigma_p \times U(1)))
}
\]
The composite $F^*(BU(p)) \to F^*(B(\Sigma_p \wr U(1))) \to
F^*(BU(1)^p)^{\Sigma_p}$ is an isomorphism. The latter map is the edge
morphism in the above spectral sequence, and so the line $s=0$
consists of permanent cycles.

As a module acted on by the group $C_p \subset \Sigma_p$, the
ring $F^*(BU(1)^p) \cong F^*\pow{\alpha_1,\ldots,\alpha_p}$ is a direct sum
of two submodules: the subring $F^*\pow{c_p}$ generated by the
monomials $(\prod \alpha_i)^k$, and a free $C_p$-module with no higher
cohomology. Therefore, for $s > 0$ the map
\[
H^s(\Sigma_p, F^*\pow{c_p}) \to H^s(\Sigma_p, F^*(BU(1)^p))
\]
is an isomorphism. The composite $F^*\pow{c_p} \to F^*(BU(1)^p) \to
F^*(BU(1))$ induces an injection on cohomology. The above spectral
sequences are, in positive cohomological degree, the tensor products
of this injective map of groups (which consist of permanent cycles)
with the cohomology spectral sequence for $F^*(B\Sigma_p)$, and so
converge to an injective map.
\end{proof}

\begin{cor}
For any complex vector bundle $\xi \to B(\Sigma_p \wr U(1))$, two
Thom classes for $\xi$ are equivalent if and only if their
restrictions to $BU(1)^p$ and $B\Sigma_p \times BU(1)$ are
equivalent.
\end{cor}

\begin{proof}
The product of the restriction maps on the $F$-cohomology of Thom
spaces is injective by naturality of the Thom isomorphism.
\end{proof}

\begin{prop}
If $F^*$ is torsion-free, the map
\[
F^*(B\Sigma_p \times BU(1)) \to F^*(BU(1)) \times
F^*(B\Sigma_p \times BU(1))/I_{tr}
\]
is injective.
\end{prop}

(This is similar to results from
\cite{hopkins-kuhn-ravenel-character}, though here we do not assume
that the coefficient ring $F^*$ is local.)

\begin{proof}
The natural K\"unneth isomorphisms on the skeleta $\mb{CP}^k \subset
BU(1)$ take the form
\[
F^*(B\Sigma_p \times \mb{CP}^k) \cong F^*(B\Sigma_p) \otimes_{F^*}
F^*(\mb{CP}^k).
\]
This inverse system in $k$ also satisfies the Mittag--Leffler
condition, and so it suffices to show that the map $F^*(B\Sigma_p) \to
F^* \times F^*(B\Sigma_p)/I_{tr}$ is injective. As the cyclic group
$C_p \subset B\Sigma_p$ has index relatively prime to $p$, the
left-hand map is injective in the commutative diagram
\[
\xymatrix{
F^*(B\Sigma_p) \ar@{^(->}[d] \ar[r] &
F^* \times F^*(B\Sigma_p)/I_{tr} \ar[d] \\
F^*(BC_p) \ar[r] &
F^* \times F^*(BC_p)/I_{tr}.
}
\]
It therefore suffices to prove that the bottom map is injective.

As $p$ is not a zero divisor in $F^*$, the $p$-series $[p]_{\mb G}(x)$
is not a zero divisor in $F^*(BU(1))$, and so the cohomology ring
$F^*(BC_p)$ is the quotient $F^*\pow{c_1} / [p]_{\mb G}(c_1)$
\cite{hopkins-mahowald-sadofsky}. The kernel of the map $F^*(BC_p) \to
F^*$ is generated by $c_1$, while the transfer ideal is generated by
the divided $p$-series $\langle p\rangle_{\mb G}(c_1) = [p]_{\mb
  G}(c_1)/c_1$ by naturality of the map $MU^* \to F^*$
\cite[4.2]{quillen-elementaryproofs}. The intersection of the ideals
$(c_1)$ and $(\langle p\rangle_{\mb G}(c_1))$ in the power series ring
consists of elements $g(c_1)\cdot \langle p\rangle_{\mb G}(c_1)$ such
that the constant coefficient of $g$ is annihilated by the constant
coefficient $p$ of $\langle p\rangle_{\mb G} (c_1)$. As $F^*$ is
torsion-free, this ideal is generated by $[p]_{\mb G}(c_1)$.
\end{proof}

\begin{cor}
\label{cor:torsionfreecheck}
For any complex vector bundle $\xi \to B\Sigma_p \times BU(1)$, two
orientations for $\xi$ are equivalent if and only if their
restrictions to $BU(1)$ are equal and their images in
$F^*(B\Sigma_p)/I_{tr} \otimes_{F^*(B\Sigma_p)} F^*(\Th(\xi))$ are equal.
\end{cor}

\section{Orientations by $MX_p$}

In this section, we fix a $p$-local $E_\infty$ ring spectrum $E$ such
that $E^*$ is torsion-free, together with a complex orientation $u$
defined by a map $MX_1 \to E$. (We will continue to write $\mb G$ for
the associated formal group law.) In this section we will analyze the
obstruction to $p$-local maps from $MX_p$.

As in Section~\ref{sec:orientation-towers}, the space of extensions of
the complex orientation to an $E_\infty$ ring map $MX_p \to E$ is the
space of orientations of the Thom spectrum over $F_p$. The
homotopy pushout diagram for $F_p$ from equation~(\ref{eq:pthspace})
expresses the map $F_p \to B_1 \to \Pic(E)$ as a coherently
commutative diagram
\[
  \xymatrix{
    B(\Sigma_p \wr U(1)) \ar[r] \ar[d] &
    BU(p) \ar[d]^{\gamma_p} \\
    B\Sigma_p \ar[r]_-{\rho} & \Pic(E).
  }
\]
From diagram~(\ref{eq:powerconstruction}), the map $BU(p) \to \Pic(E)$
classifies the Thom spectrum $M_E \gamma_p = E \wedge M\gamma_p$
associated to the tautological bundle $\gamma_p$ of $BU(p)$, while the
map $B\Sigma_p \to \Pic(E)$ classifies the bundle $M_E \rho$
associated to the regular representation $\rho\co \Sigma_p \to
U(p)$. An orientation of the resulting Thom spectrum over $F_p$ exists
if and only if there are orientations of $M_E \gamma_p$ and $M_E \rho$
whose restrictions to $B(\Sigma_p \wr U(1))$ agree.

\begin{prop}
\label{prop:orientation-hardcheck}
There exists an orientation of the Thom spectrum over $F_p$ if and
only if the orientations $t_u(\rho)$ and $t_u(\gamma_p)$ have the same
restriction to $B(\Sigma_p \wr U(1))$, or equivalently if
\[
t_u(\rho) \circ \Pp^{\wedge_E}(\gamma_1) = t_u(\Pp(\gamma_1)).
\]
\end{prop}

\begin{proof}
The ``if'' direction is clear. In the other direction, we start by
assuming that we have some pair of orientations whose restrictions
agree.

Any orientation of $M_E \rho$ is of the form $a \cdot t_u(\rho)$
for some $a \in E^0(B\Sigma_p)^\times$, and similarly any
orientation of $M_E \gamma_p$ is of the form $b \cdot t_u(\gamma_p)$ 
for some $b \in E^0(BU(p))^\times$. These restrict to $a \cdot
t_u(\rho) \circ \Pp^{\wedge_E}(\gamma_1)$ and $b \cdot
t_u(\Pp(\gamma_1))$ respectively.

By Corollary~\ref{cor:orientationrestriction}, the restrictions of
these orientations to $BU(1)^p$ are $\epsilon(a) \cdot
t_u(\gamma_1^{\boxplus p})$ and $b \cdot t_u(\gamma_1^{\boxplus b})$,
where $\epsilon(a)$ is the natural restriction of $a$ to
$(E^0)^\times$ and we identify $b$ with its image under the injection
$E^0(BU(p)) \to E^0(BU(1)^p)$. Similarly, the restrictions of these
orientations to $B\Sigma_p$ are $a \cdot t_u(\rho)$ and $\epsilon(b)
\cdot t_u(\rho)$ respectively.

For these to be equal as needed, both $a$ and $b$ must be in the image
of $(E^0)^\times$ and equal. Changing the orientation by multiplying
by $a^{-1}$ then gives the desired result.
\end{proof}

Combining this with Corollary~\ref{cor:orientationrestriction} and
Corollary~\ref{cor:torsionfreecheck}, we find the following.
\begin{prop}
\label{prop:orientation-simplecheck}
There exists an orientation of the Thom spectrum over $F_p$ if and
only if the orientations $t_u(\rho)$ and $t_u(\gamma_p)$ define the
same generating class after first restricting to $B\Sigma_p
\times BU(1)$ and then tensoring over $F^*(B\Sigma_p)$ with
$F^*(B\Sigma_p)/I_{tr}$.
\end{prop}

We recall for the following that, if $E^*$ is torsion-free, we have an
isomorphism
\[
E^*(B\Sigma_p) \cong E^*(BC_p)^{\mb F_p^\times} =
\left(E^*\pow{z}/[p]_{\mb G}(z)\right)^{\mb F_p^\times},
\]
where the action of $\mb F_p^\times$ is by $z \mapsto [i]_{\mb G}[z]$.
\begin{thm}
If $E$ is an $E_\infty$ ring spectrum such that $E^*$ is $p$-local and
torsion-free, an $E_\infty$ orientation $MX_1 \to E$ extends to an
$E_\infty$ orientation $MX_p \to E$ if and only if the power operation
$\Psi_u$ satisfies the Ando criterion: we must have
\[
\Psi_u(c_1) = \prod_{i=0}^{p-1} (c_1 +_{\mb G} [i]_{\mb G}(z))
\]
in the ring $E^{2pk}(\Th(\rho \boxtimes \gamma_1)) / I_{tr}$.
\end{thm}

\begin{proof}
It is necessary and sufficient, by
Proposition~\ref{prop:orientation-simplecheck}, to know that the
restrictions of $t_u(\rho) \circ \Pp^{\wedge_E}(\gamma_1)$ and
$t_u(\Pp(\gamma_1))$ to this target ring are equal. By definition, the
former restricts to $\Psi_u(t_u \gamma_1) = \Psi_u(c_1)$. The latter
restricts to $t_u(\rho \boxtimes \gamma_1)$, and so this formula for
the Thom class follows from the splitting principle.
\end{proof}

\section{Cohomology monomorphisms for $E$-theory}
\label{sec:extra}

In this section we will show the following result, which is closely
related to Proposition~\ref{prop:cohomologyrestriction} when $X =
BU(1)_+$.

\begin{prop}
  Let $E$ be a $p$-local, complex orientable $E_\infty$ ring spectrum
  whose coefficient ring has no $p$-torsion, and let $X$ be a based
  space with $p$-fold smash power $X^{(p)}$ such that $E_* X$ is a
  direct sum of (unshifted) copies of $E_*$. Then the restriction map
  \[
  E^*(X^{(p)}_{hC_p}) \to E^*(X^{(p)}) \times
  E^*((BC_p)_+ \wedge X)
  \]
  is a monomorphism.
\end{prop}

For instance, this is satisfied when $E$ is Morava $E$-theory and $X$
is of finite type with $\mb Z_{(p)}$-homology concentrated in even
degrees. This allows us to remove the finite type hypothesis from the
cohomology theory in \cite[VIII.7.3]{bmms-hinfty} so that it applies
to Morava $E$-theory (e.g. see \cite[4.4.2]{ando-isogenies},
\cite[proof of~6.1]{ando-hopkins-strickland-sigma}). We would like to
thank Eric Peterson and Nathaniel Stapleton for bringing this to our
attention.

\begin{proof}
  Write $M = F(\Sigma^\infty X^{(p)}, E)$ for the function spectrum,
  which has an action of $C_p$ from the source, and $N =
  F(\Sigma^\infty X,E)$, with the trivial action of $C_p$. The
  homotopy fixed-point map
  \[
  M^{hC_p} \to M,
  \]
  on homotopy groups, becomes the map $E^*(X^{(p)}_{hC_p}) \to
  E^*(X^{(p)})$. On the other hand, the map
  \[
  M^{hC_p} \to N^{hC_p}
  \]
  becomes the map $E^*(X^{(p)}_{hC_p}) \to E^*((BC_p)_+ \wedge X)$.
  We want to prove that these are jointly monomorphisms.

  By the assumptions on $X$, we have that $E \wedge X \simeq
  \bigvee_{\alpha} E$ as $E$-modules, and so there is a
  $C_p$-equivariant equivalence $E \wedge X^{(p)} \cong (E \wedge
  X)^{\wedge_E p}$. Using the decomposition of $E \wedge X$ into a
  wedge of copies of $E$, this decomposes $C_p$-equivariantly as an
  $E$-module into a $C_p$-fixed component and a $C_p$-free component:
  \[
  (E \wedge X)^{\wedge_E (p)} \cong \left(\bigvee_\alpha E\right)
  \vee \left(\bigvee_\beta (C_p)_+ \wedge E\right).
  \]
  The spectrum $M$ is $E$-dual to this, and we calculate
  \[
  \pi_* M^{hC_p} \cong \left(\prod_\alpha E_*\pow{x}/ [p]_F(x)\right)
  \times \prod_\beta E_*.
  \]
  We find that the map $M^{hC_p} \to M$ is a monomorphism on the right-hand
  factor. On the left-hand factor coming from the terms with trivial
  action, it becomes a product of projection maps
  \[
  \prod_\alpha E_*\pow{x} / [p]_F(x) \to \prod_\alpha E_*
  \]
  which send $x$ to zero. The kernel of this consists precisely of
  the multiples of $x$. Therefore, to finish the proof, we simply need
  to show that the multiples of $x$ map monomorphically into the
  homotopy of $N^{hC_p}$.

  We now consider, for any finite subcomplex $Y \subset X$, the
  natural diagram of homotopy fixed-point and Tate spectra:
  \[
  \xymatrix{
    F(X^{(p)}, E)^{hC_p} \ar[r] \ar[d] & 
    F(X, E)^{hC_p} \ar[d] \\
    F(X^{(p)}, E)^{tC_p} \ar[r] \ar[d] & 
    F(X, E)^{tC_p} \ar[d] \\
    F(Y^{(p)}, E)^{tC_p} \ar[r]& 
    F(Y, E)^{tC_p}.
  }
  \]
  The upper left-hand map is the localization
  \[
  \left(\prod_\alpha E_*\pow{x} / [p]_F(x)\right) \times \prod_\beta E_* \to x^{-1}
  \prod_\alpha E_*\pow{x}/ [p]_F(x),
  \]
  which is a monomorphism on the multiples of $x$. 

  The bottom map is a natural transformation of functors in $Y$, and
  is evidently an equivalence when $Y$ is $S^0$. Both of these
  functors take cofiber sequences of based spaces $Y$ to fiber
  sequences of spectra, as follows. For a cofiber sequence $Y' \to Y
  \to Y''$, the smash power of $Y$ has an equivariant filtration whose
  $k$'th associated graded consists of smash products of $k$ copies of
  $Y'$ and $(p-k)$ copies of $\Sigma Y''$; the terms other than $k =
  0$ and $k=p$ are acted on freely by $C_p$ and do not contribute to
  the Tate spectrum. Therefore, the bottom map is an equivalence for
  any finite complex $Y$.

  Given our space $X$, we observe that
  \[
  \colim_Y E_* Y \cong E_* X \cong \oplus_\alpha E_*
  \]
  as $Y$ ranges over the filtered system of finite subcomplexes of
  $X$.  Therefore, for any index $\alpha$ the corresponding generator
  of $E_* X$ lifts to $E_* Y$ for some such $Y$. Any nonzero element
  in $\pi_* F(X^{(p)},E)^{tC_p}$ then has nonzero restriction to
  $\pi_* F(Y^{(p)}, E)^{tC_p}$ for some $Y$, and hence (since the
  bottom map is an isomorphism) has nonzero image in $\pi_*
  F(X,E)^{tC_p}$. Thus, the center map is injective.

  Since the center map is injective and the upper-left map is
  injective on multiples of $x$, the top map in the diagram is also
  injective on multiples of $x$ as desired.
\end{proof}

\bibliography{../masterbib}

\end{document}